\documentclass[a4paper,10pt,twoside]{article}
\usepackage{definitions/CMST}
\usepackage[ps,arrow,matrix,tips,curve]{xy}
\usepackage{amsmath,amsthm}
\usepackage{booktabs}

\newtheorem{theorem}{Theorem}
\newtheorem{corollary}{Corollary}[theorem]
\newtheorem{lemma}{Lemma}[theorem]

\newcommand\addTitle{$3n + 3^k$ Problem}
\newcommand\addShortTitle{$3n + 3^k$ Problem}

\newcommand\addAuthors{D. Barina$^{1*}$, W.C. Maat$^{2}$}
\newcommand\addAuthorsNoSuperscript{D. Barina, W.C. Maat}
\newcommand\addAffiliations{%
$^1$Brno University of Technology\\
Faculty of Information Technology\\
Bozetechova 2, 638 00 Brno\\
$^*$E-mail: ibarina@fit.vutbr.cz\\[5pt]

$^2$(independent researcher)\\
E-mail: willem@ai-brain.net
}

\newcommand\addAbstract{The Collatz problem is generalized into $3n + 3^k$ problem.
It is shown that as long as the Collatz function iterates converge to the cycle passing through the number 1, the $3n + 3^k$ sequence converges to the cycle passing through the number $3^k$ for arbitrary positive integers $n$ and $k$.
The proof shows that the sequence of $3n + 3^k$ function iterates for a number $3^k n$ is exactly the sequence of the Collatz function iterates for $n$ multiplied by~$3^k$.}
\newcommand\addKeywords{Collatz conjecture, number theory, dynamical systems}

\begin{document}
\maketitle
\begin{multicols}{2}

Collatz problem \cite{Lagarias2003,Lagarias2006} is number theory problem that provides an algorithm for generating a sequence.
The algorithm is as follows.
Start with arbitrary positive integer $n$.
If $n$ is even, divide it by two.
Else if $n$ is odd, multiply it by three and add one.
The conjecture says that we will always arrive at the number 1.

More formally, the Collatz conjecture asserts that a sequence defined by repeatedly applying the function
\begin{align}
	T_0(n) &= \begin{cases}
			(3n + 1) / 2, & \text{for odd $n$,}\\
			n / 2,        & \text{for even $n$,}\\
		\end{cases}
\end{align}
always converges to the cycle passing through the number 1 for arbitrary positive integer $n$.

\begin{theorem}
As long as the $T_0$ iterates converge to 1,
a sequence defined by repeatedly applying the function
\begin{align}
	T_k(n) &= \begin{cases}
			(3n + 3^k) / 2, & \text{for odd $n$,}\\
			n / 2,          & \text{for even $n$,}\\
		\end{cases}
\end{align}
converges to the cycle passing through the number $3^k$ for arbitrary positive integers $n$ and $k$.
\end{theorem}

Note that the $3n + 3^k$ problem was recently discussed in \cite{Boulkaboul2022}.

\begin{proof}
The function $T_{k}(n)$ can be adjusted for multiples of three using
\begin{align} \label{eqn:3T0n}
	3 \cdot T_{k}(n) = \begin{cases}
			(3 \cdot 3n + 3^k \cdot 3) / 2, & \text{for odd $n$,}\\
			3n/2,                           & \text{for even $n$.}\\
		\end{cases}
\end{align}
Now substitute $3n$ for $m$ ($m$ is a multiple of 3), thus ${3 \cdot T_{k}(n) = T_{k+1}(3n)}$, and therefore
\begin{align} \label{eqn:T1m}
	T_{k+1}(m) = \begin{cases}
			(3m + 3^{k+1}) / 2, & \text{for odd $m$,}\\
			m/2,                & \text{for even $m$.}\\
		\end{cases}
\end{align}
Note that $3n$ is odd when $n$ is odd and $3n$ is even when $n$ is even.
It means that the sequence of $T_{k+1}$ iterates for $m=3n$ is exactly the sequence of $T_{k}$ iterates for $n$ multiplied by 3.

Now we know what happens to the trajectory of $T_{k+1}$ if $m$ is a multiple of 3.
It remains to show what happens to this trajectory if $m$ is not a multiple of 3.
If such a number is even, then we can repeatedly pull out all the factors of 2 (the even branch of $T_{k+1}$), and finally get an odd number.
Thus we focus on odd $m$.
A single iterate of the $T_{k+1}$ function give us $(3m + 3) / 2$ (the odd branch of $T_{k+1}$), where $(3m + 3) / 2$ is a multiple of 3 ($3m+3$ is obviously multiple of 3, the division by $2$ has no effect on divisibility by 3).
Note that the iterates of $T_{k+1}$ converge to the cycle passing through the number $3^{k+1}$, which corresponds to $T_0(n) = 1$.
\end{proof}

\begin{corollary}
	$T_k$ iterates always lead to the cycle $1 \cdot 3^k \rightarrow 2 \cdot 3^k \rightarrow 1 \cdot 3^k$, assuming the Collatz conjecture holds.
\end{corollary}

The iterates of $T_0$ lead to the cycle $1 \rightarrow 2 \rightarrow 1$, where as the iterates of $T_1$ go to $3 \rightarrow 6 \rightarrow 3$, etc.

\begin{corollary}
The sequence of $T_k$ iterates for a number $3^k n$ is exactly the sequence of $T_0$ iterates for $n$ multiplied by $3^k$.
\end{corollary}

In short, $T_k(3^k n) = 3^k T_0(n)$.
See an example for $T_0$, $T_1$ and $T_2$ in Table~\ref{tab:1}.

\begin{table*}[htp]
	\centering
	\footnotesize
	\setlength{\tabcolsep}{2.5pt}
	\begin{tabular}{l@{\hspace{6pt}} *{7}{c}}
		\toprule
		\bfseries Function & \multicolumn{7}{c}{\bfseries Iteration} \\
		\cmidrule(l){2-8}
		& 0 & 1 & 2 & 3 & 4 & 5 & 6 \\
		\midrule
		\bfseries $T_2$
		& 189 & 288 & 144 & 72 & 36 & 18 & 9 \\
		\bfseries $T_1$
		&  63 &  96 &  48 & 24 & 12 &  6 & 3 \\
		\bfseries $T_0$
		&  21 &  32 &  16 &  8 &  4 &  2 & 1 \\
		\bottomrule
	\end{tabular}
	\caption{$T_0$, $T_1$ and $T_2$ function iterates.}%
	\label{tab:1}
\end{table*}

Let's start $T_1$ iterations with an odd starting number $n$.
The first $T_1$ take us to $(3n+3)/2$, which is a number of the form $3m/2$.
Meanwhile take the even number $n+1$ and subject it to iteration $T_0$.
The result is $(n+1)/2$ (since the $n+1$ was even), which is the number of the form $m/2$.
We can see that further $T_1$ iterates are exactly the $T_0$ iterates multiplied by 3.
Thus, there is a relationship between the initial number $n$ for $T_1$ and the number $n+1$ for $T_0$.
The example for numbers 63 and 64 is given in Table~\ref{tab:2}.


\begin{table*}[htp]
	\centering
	\footnotesize
	\setlength{\tabcolsep}{2.5pt}
	\begin{tabular}{l@{\hspace{6pt}} *{7}{c}}
		\toprule
		\bfseries Function & \multicolumn{7}{c}{\bfseries Iteration} \\
		\cmidrule(l){2-8}
		& 0 & 1 & 2 & 3 & 4 & 5 & 6 \\
		\midrule
		\bfseries $T_1$
		& 63 & 96 & 48 & 24 & 12 & 6 & 3 \\
		\bfseries $T_0$
		& 64 & 32 & 16 &  8 &  4 & 2 & 1 \\
		\bottomrule
	\end{tabular}
	\caption{$T_0$ and $T_1$ function iterates for $n+1$ and $n$ starting values.}%
	\label{tab:2}
\end{table*}

In \cite{Lagarias1990} 
Lagarias extends $T_0(n)$ to the rational numbers with odd denominator. And he shows that any rational element $q$ of cycle of $T_0(n)$ must be of the form
\begin{align}
q=\frac{\sum_{i=0}^{l-1}3^i2^{a_i}}{2^m-3^{l}}
\end{align}
for some $l\geq0$, $m>l$ and $m>a_0>a_1>\ldots>a_{l-1}\geq0$.

\begin{lemma}
The map $T_0(n)$ has no non-zero rational cycle elements with a denominator which is a multiple of $3$.
\end{lemma}
\begin{proof}
Suppose that $l=0$ then $2^m-3^0=2^m-1$ can be a multiple of $3$ (for instance for $m=4$) however in this case the sum in the numerator is $0$, because we sum from $i=0$ to $-1$ which is an empty sum and therefore $q=0$.

Now suppose that $l>0$ then the numerator is $2^m-3^l\equiv 2^m\pmod{3}$ and $2^m$ is never a multiple of $3$ by the
unique decomposition of numbers into prime numbers.
\end{proof}

Another way of looking at $T_{k}(n)$ is by introducing the function $L_{3^k}(n)=n/3^k$. Notice that $L_{3^k}(n)$ is a bijection on the rational numbers with the inverse function ${L_{3^k}}^{-1}(n)=3^kn$. 

We have that
\begin{align}
T_{k}(n)= {L_{3^k}}^{-1}\circ T_0 \circ L_{3^k}(n)
\end{align}
let us show this.
\begin{align}
\begin{aligned}
\label{TkL}
T_{k}(n)&=\begin{cases} 
(3n + 3^k) / 2, & \text{for odd $n$,}\\
n / 2,        & \text{for even $n$,}\\
\end{cases}\\
&=\begin{cases} 
3^k(3(n/3^k) + 1) / 2, & \text{for odd $n$,}\\
3^k(n/3^k) / 2,        & \text{for even $n$,}\\
\end{cases}\\
&={L_{3^k}}^{-1}\circ T_0 \circ L_{3^k}(n)
\end{aligned}
\end{align}

If we look at our equation in (\ref{TkL}) then we have the following commutative diagram
\[
\xymatrix{n/3^k\ar[r]^{T_0}&T_0(n/3^k)\ar[d]^{{L_{3^k}}^{-1}}\\
n\ar[r]_{T_k}\ar[u]^{L_{3^k}} & T_k(n)}
\]
If we use the notation $\mathbb{Q}[(2)]$ for all rational numbers with odd denominator and we iterate the above diagram then we get the following diagram
\[
\xymatrix{
\mathbb{Q}[(2)]\ar[r]^{T_0}& \mathbb{Q}[(2)]\ar[d]_{{L_{3^k}}^{-1}}\ar[r]^{id}&\mathbb{Q}[(2)]\ar[r]^{T_0}&\mathbb{Q}[(2)]\ar[d]_{{L_{3^k}}^{-1}} \\
\mathbb{Q}[(2)]\ar[u]^{L_{3^k}}\ar[r]_{T_k}  &\mathbb{Q}[(2)]\ar[ru]_{L_{3^k}}\ar[rr]_{T_k}&&\mathbb{Q}[(2)]
}
\]
giving us that 
\begin{align}
T_k(T_k(n))={L_{3^k}}^{-1}(T_0(T_0(L_{3^k}(n))))
\end{align}
and we can keep doing this, because $L_{3^k}$ and ${L_{3^k}}^{-1}$ cancel eachother out. We will use this extensively in the proof of the next corollary.

\begin{corollary}
$T_k(n)$ has exactly the same integer cycles as $T_0(n)$ only multiplied by $3^k$.
\end{corollary}
\begin{proof}
Let $q$ be a element of an integral cycle of $T_k(n)$ of length $l$ then
\begin{align}
\begin{aligned}
q&={T_k}^{(l)}(q)\\
&={L_{3^k}}^{-1}\circ T_0\circ L_{3^k}\circ {L_{3^k}}{-1}\circ T_0\circ L_{3^k}\circ\\
&\ \ \ \ \ \underbrace{\ldots\circ {L_{3^k}}^{-1}\circ T_0\circ L_{3^k}(q)}_{l}\\
&={L_{3^k}}^{-1}\circ \underbrace{T_0\circ\ldots\circ T_0}_{l}\circ L_{3^k}(q)\\
&={L_{3^k}}^{-1}\circ {T_0}^{(l)}\circ L_{3^k}(q)
\end{aligned}
\end{align}
where the ${L_{3^k}}^{-1}\circ T_0\circ L_{3^k}$ is composed exactly $l$ times on the second and third line.
By applying $L_{3^k}$ from the left on the first part of the equation above and the last part of the equation gives us:
\begin{align}
q/3^k = {T_0}^{(l)}(q/3^k)
\end{align}
so $q/3^k$ is a rational cycle element of $T_0$ however by the previous there are no rational cycle elements where the denominator is a multiple of $3$. Therefore we have that $3^k|q$ and hence $q=3^kq'$. Now we have that $q'$ is a fixed point of ${T_0}^{(l)}$ since
\begin{align}
\begin{aligned}
q/3^k &= {T_0}^{(l)}(q/3^k)\\
3^kq'/3^k&={T_0}^{(l)}(3^kq'/3^k)\\
q'&={T_0}^{(l)}(q')\
\end{aligned}
\end{align}
so any integral cycle element of $T_k$ is a integral cycle element of $T_0$ multiplied by $3^k$.

Now suppose that $q$ is an integral cycle element of $T_0$ then we claim that ${T_{k}}^{(l)}(3^kq)=3^kq$. We have seen that
\begin{align}
{T_k}^{(l)}(n)={L_{3^k}}^{-1}\circ {T_0}^{(l)}\circ L_{3^k}(n)
\end{align}
hence
\begin{align}
\begin{aligned}
{T_k}^{(l)}(3^kq)&={L_{3^k}}^{-1}\circ {T_0}^{(l)}\circ L_{3^k}(3^kq)\\
&={L_{3^k}}^{-1}\circ {T_0}^{(l)}(3^kq/3^k)\\
&={L_{3^k}}^{-1}\circ {T_0}^{(l)}(q)\\
&={L_{3^k}}^{-1}(q)\\
&=3^kq
\end{aligned}
\end{align}
\end{proof}

In \cite{Lagarias1990} a special type of generalization of $T_0(n)$ is studied. Namely
\begin{align}
\begin{aligned}
	T_O(n) &= \begin{cases}
			(3n + O) / 2, & \text{for odd $n$,}\\
			n / 2,        & \text{for even $n$,}\\
		\end{cases}
\end{aligned}
\end{align}
where $O\equiv\pm1\pmod{6}$. So $O$ is either of the form $O=6o+1$ or $O=6o-1$ for some integer $o$. So in particular $O$ will never be a multiple of $3$. We will extend the result from the previous corollary to this generalization $T_O(n)$.

\begin{theorem}
The integral cycles of the map $T_{3^kO}(n)$ are exactly the integral cycles of $T_O(n)$ multiplied by $3^k$.
\end{theorem}
\begin{proof}
First note that for any non-zero integers $A$ and $B$ we have 
\begin{align}
\begin{aligned}
L_A(L_B(n))&=n/B/A=n/(B\cdot A)\\
&=n/(A\cdot B)=n/A/B=L_B(L_A(n))
\end{aligned}
\end{align}
likewise we have ${L_B}^{-1}({L_A}^{-1}(n))={L_A}^{-1}({L_B}^{-1}(n))$.

Just as in the case of $T_k(n)$ we have $T_{O}(n)={L_{O}}^{-1}(T_0(L_O(n)))$ and $T_{3^kO}(n)={L_{3^kO}}^{-1}(T_0(L_{3^kO}(n)))$. And since $L_{3^kO}(n)=L_{3^k}(L_O(n))$ we have that
\begin{align}
\begin{aligned}
T_{3^kO}^{(l)}(n)&={L_{3^kO}}^{-1}\circ T_0\circ L_{3^kO}\circ {L_{3^kO}}^{-1}\circ T_0\circ L_{3^kO}\circ\\
&\ \ \ \ \ \underbrace{\ldots\circ {L_{3^kO}}^{-1}\circ T_0\circ L_{3^kO}(n)}_{l}\\
&={L_{3^k}}^{-1}\circ {L_{O}}^{-1}\circ T_0\circ L_O \circ L_{3^k}\circ\\
&\ \ \ \ \ \underbrace{\ldots\circ {L_{3^k}}^{-1}\circ{L_O}^{-1}\circ T_0\circ L_O\circ L_{3^k}}_{l}(n)\\
&={L_{3^k}}^{-1}\circ T_O\circ L_{3^k}\circ {L_{3^k}}^{-1}\circ T_O\circ L_{3^k}\circ\\
&\ \ \ \ \ \underbrace{\ldots\circ {L_{3^k}}^{-1}\circ T_O \circ L_{3^k}}_{l}(n)\\
&={L_{3^k}}^{-1}\circ\underbrace{T_O\circ\ldots T_O}_{l}\circ L_{3^k}(n)\\
&={L_{3^k}}^{-1}\circ {T_O}^{(l)}\circ L_{3^k}(n)
\end{aligned}
\end{align}
Where the compositions on the first and second are exactly $l$ times, the compositions on the third and fourth line are exactly $l$ times and the compositions on lines five and six are $l$ times.
Likewise we have 
\begin{align}
{T_O}^{(l)}(n)={L_{O}}^{-1}({T_k}^{(l)}(L_O(n)))
\end{align}
and 
\begin{align}
{T_O}^{(l)}(n)={L_{3^kO}}^{-1}({T_0}^{(l)}(L_{3^kO}(n)))
\end{align}
Now suppose that $q$ is a integral cycle element of $T_{3^kO}(n)$ and let $l$ be the length of this cycle. Then we have that
\begin{align}
\begin{aligned}
q&={L_{3^k}}({T_O}^{(l)}(L_{3^k}(q))) \Leftrightarrow\\
q/3^k&={T_O}^{(l)}(q/3^k)
\end{aligned}
\end{align}
hence $q/3^k$ is a (fractional) cycle element of $T_O(n)$. But we also have that 
\begin{align}
\begin{aligned}
q/3^k={T_O}^{(l)}(q/3^k)&={L_O}^{-1}(T_0^{(l)}(L_O(q/3^k))) \Leftrightarrow\\
q/(3^kO)&={T_0}^{(l)}(q/(3^kO))
\end{aligned}
\end{align}
hence we have that $q/(3^kO)$ is a (rational) cycle element of $T_0(n)$. However we have seen that there are no rational cycle elements of $T_0(n)$ with a denominator divisible by $3$. Therefore $q=3^k\cdot q'$ and $q'/O$ is a (rational) cycle element of $T_0(n)$. And therefore $q/3^k=3^k\cdot q'/3^k = q'$ is an integral cycle element of $T_O(n)$.

To complete the proof we have 
\begin{align}
\begin{aligned}
{T_{3^kO}}^{(l)}(q)&={L_{3^k}}({T_O}^{(l)}(q/3^k))\\
&=3^k\cdot{T_O}^{(l)}(q')\\
&=3^k\cdot q' = q
\end{aligned}
\end{align}
so any cycle element of $T_{3^kO}(n)$ is $3^k$ times a cycle element of $T_O(n)$. And suppose that $q'$ is a cycle element of $T_O(n)$ then we have
\begin{align}
\begin{aligned}
{T_{3^kO}}^{(l)}(3^k\cdot q')&={L_{3^k}}^{-1}({L_O}^{(l)}(L_{3^k}(3^k\cdot q')))\\
&=3^k\cdot{L_O}^{(l)}(3^k\cdot q'/3^k)\\
&=3^k\cdot{L_O}^{(l)}(q')\\
&=3^k\cdot q'
\end{aligned}
\end{align}
so then $3^k\cdot q'$ is a integral cycle element of $T_{3^kO}(n)$.
\end{proof}

\subsection*{Conclusion}
We have seen that $T_k(n)$ behaves on multiples of $3^k$ exactly as $T_0(n)$ and we have also seen that any $n$ becomes a multiple of $3^k$ after finitely many steps.
Therefore studying $T_k(n)$ gives us no new information about $T_0(n)$.
However it could be the case that someone is able to prove the conjecture for $T_k(n)$ which automatically means that they prove the conjecture for $T_0(n)$. Likewise we have seen that $T_{3^kO}(n)$ behaves on multiples of $3^k$ exactly like $T_O(n)$ and that they have the same integral cycles (up to a factor $3^k$). So all heuristics for $T_O(n)$ also hold for $T_{3^kO}(n)$.

\titleformat{\section}[block]{\filright\bfseries}{\thesection}{1em}{}
\section*{Acknowledgment}\affSettings
This work was supported by the Ministry of Education, Youth and Sports of the Czech Republic through the e-INFRA CZ (ID:90254).

\titleformat{\section}[block]{\filright\bfseries}{\thesection}{1em}{}
\bibliographystyle{plain}
\bibliography{sources}
\end{multicols}

\clearpage
\vspace{1cm}

\begingroup
\bioSettings
\noindent \begin{tabular}{ l  p{13.4cm}}
\parbox[l]{3.0cm}{\fcolorbox{gray}{white}{\includegraphics[width=3cm]{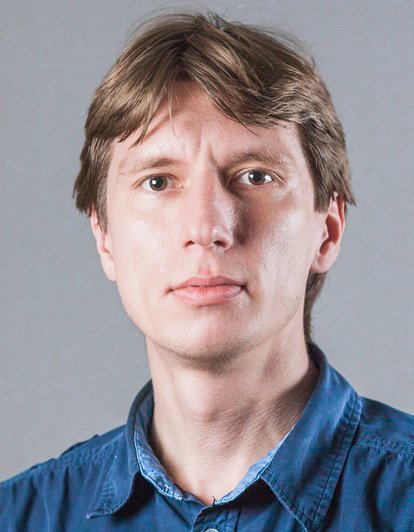}}} &  
\parbox{13.7cm}{{\bf David Barina} 
received the PhD degree at the Faculty of Information Technology, Brno University of Technology, Czech Republic. He is currently a member of the Graph@FIT group at the Department of Computer Graphics and Multimedia at FIT, Brno University of Technology. His research interests include wavelets and fast algorithms in signal and image processing.
}\\

\vspace*{0.2cm}\\
\parbox[l]{3.0cm}{\fcolorbox{gray}{white}{\includegraphics[width=3cm]{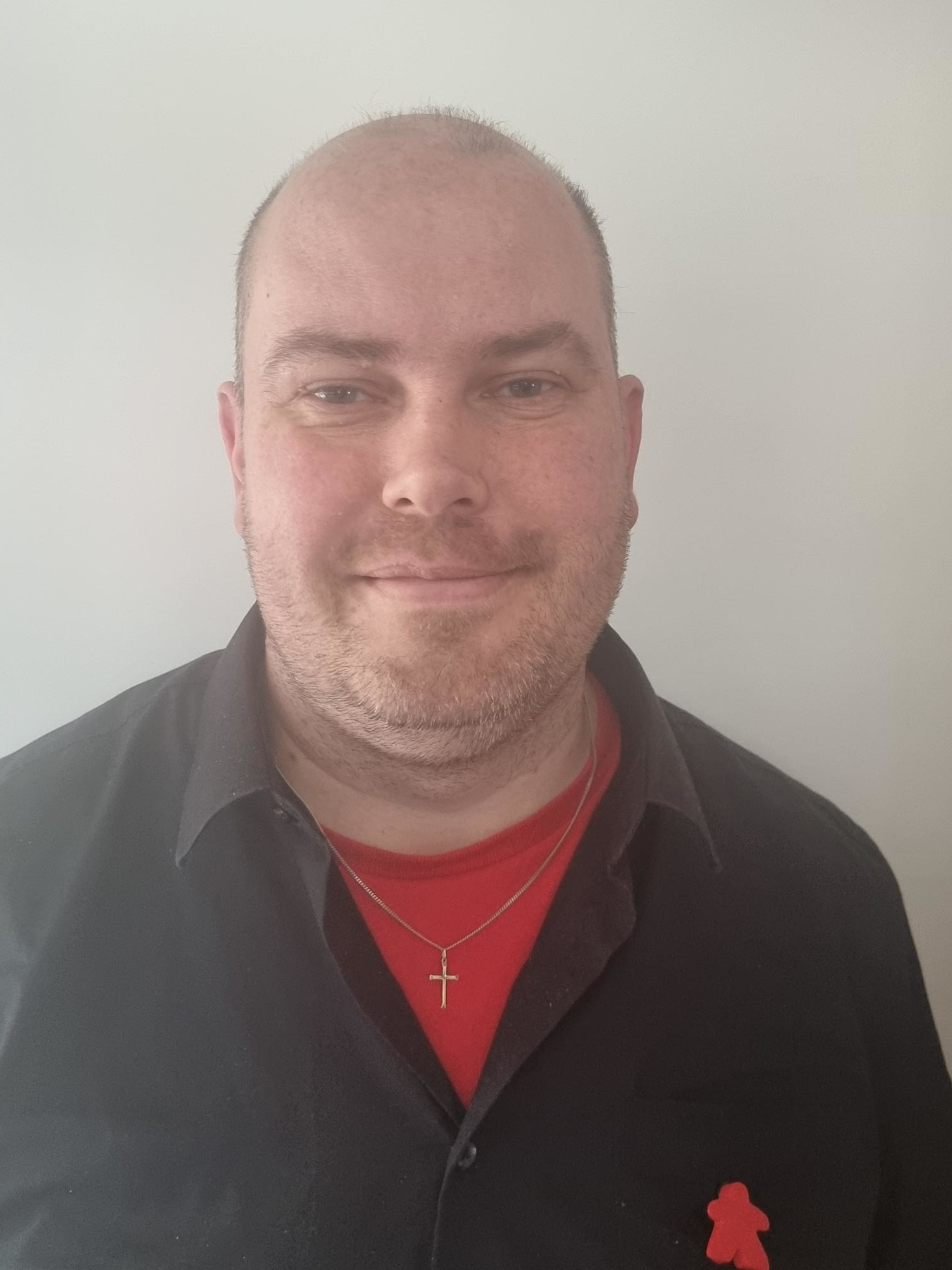}}} &  
\parbox{13.7cm}{{\bf Willem Maat} 
received the master's degree at the Faculty of Mathematics and Information Technology, Utrecht University, The Netherlands. He currently works for an insurance company as a computer programmer and does some mathematics in his spare time.
}
\end{tabular}
\endgroup
\end{document}